\documentclass{amsart}
\usepackage{amsmath, amsfonts, amssymb, latexsym}

\newtheorem{theorem}{Theorem}
\newtheorem{assumption}[theorem]{Assumption}

\newtheorem{definition}[theorem]{Definition}
\newtheorem{example}[theorem]{Example}

\newtheorem{lemma}[theorem]{Lemma}

\newtheorem{remark}[theorem]{Remark}

\begin{document}
\title{Strong Law of Large Numbers for  branching diffusions}
\author{J\'anos Engl\"ander, Simon C. Harris and Andreas E. Kyprianou}
\date{\today}
\address{Department of Statistics and Applied Probability\\
University of California, Santa Barbara, CA 93106-3110, USA}
\email{englander@pstat.ucsb.edu}\address{Department of Mathematical Sciences, University
of Bath, Bath, BA2 7AY, UK}\email{S.C.Harris@bath.ac.uk}\address{Department of Mathematical Sciences, University of Bath, Bath, BA2 7AY, UK}\email{a.kyprianou@bath.ac.uk}


\begin{abstract}
Let $X$ be the branching particle diffusion corresponding to the operator $Lu+\beta (u^{2}-u)$ on $D\subseteq \mathbb{R}^{d}$ (where $\beta  \geq 0$ and $\beta\not\equiv 0$). Let $\lambda _{c}$ denote the generalized principal eigenvalue for the operator $L+\beta $ on $D$ and assume that it is finite. When $\lambda _{c}>0$ and $L+\beta-\lambda _{c}$ satisfies certain spectral theoretical conditions, we prove that the random measure $\exp \{-\lambda _{c}t\}X_{t}$ converges almost surely in the vague topology as $t$ tends to infinity. This result is motivated by a cluster of articles due to Asmussen and Hering dating from the mid-seventies as well as the more recent work concerning analogous results  for superdiffusions of  \cite{ET,EW}. We extend significantly the results in \cite{AH76,AH77} and include some key examples of the branching process literature. As far as the proofs are concerned, we appeal to modern techniques concerning martingales and `spine' decompositions or `immortal particle pictures'.
\end{abstract}
\maketitle
\tableofcontents
\section{Introduction and statement of results}
\subsection{Model}

Write $C^{i,\eta }(D)$ to denote the space of $i$ times ($i=1,2$) continuously differentiable functions with all their $i$th order derivatives belonging to $C^{\eta }(D)$. [Here $C^{\eta }(D)$ denotes the usual
H\"{o}lder space.] Let $D\subseteq \mathbb{R}^{d}$ be a domain and consider $Y=\{Y_t;\,t\geq 0\}$, the diffusion process with probabilities $\{\mathbb{P}{_{x}},\ x\in D\}$ corresponding to the operator
\begin{equation}
L\;=\;\frac{1}{2}\nabla \!\cdot \!a\nabla \,+\,b\!\cdot \!\nabla \quad \text{%
on}\ \;\mathbb{R}^{d},  \label{notation.L}
\end{equation}
where the coefficients $a_{i,j}$ and $b_{i}$ belong to $\,C^{1,\eta },$ $%
\,i,j=1,...,d,$\thinspace\ for some $\eta $ in $(0,1],$ and the symmetric
matrix $a=\{a_{i,j}\}$ is positive definite for all $x\in D$. At this point, we do not
assume that $Y$ is conservative, i.e. $Y$ may get killed at the Euclidean
boundary of $D$ or run out to infinity in finite time.

Furthermore let us first assume that $0\leq \beta \in C^{\eta }(D)$ is bounded from above on $D$ and $\beta
\not\equiv 0$. The (strictly dyadic) \emph{$(L,\beta ;D)$-branching diffusion} is the
Markov process with motion component $Y$ and with spatially dependent rate $%
\beta $, replacing particles by precisely two offspring when branching and
starting from a finite configuration of individuals. At each time $t>0$, the
state of the process is denoted by $X_{t}\in \mathcal{M}\left( D\right) $ where
\begin{equation*}
\mathcal{M}\left( D\right) =\left\{ \sum_{i=1}^{n}\delta _{x_{i}}:n\in
\mathbb{N}\text{ and }x_{i}\in D\text{ for }i=1,...,n\right\} .
\end{equation*}
We will also use the following notation: $X=\{X_t : t\geq 0\}$ has probabilities $\{P_{\mu }:\mu \in
\mathcal{M}\left( D\right) \}$, and $E_{\mu }$ is expectation with
respect to $P_{\mu }.$ As usual,
$\langle f,\mu \rangle :=\int_{D}f(x)\,\mu (\mathrm{d}x)$ and $\langle f,g \rangle :=\int_{D}f(x)g(x)\, \mathrm{d}x$, where $\mathrm{d}x$ is Lebesgue measure, and so
$\langle f,g\mathrm{d}x \rangle=\langle fg, \mathrm{d}x\rangle=\langle f,g \rangle$.

When $\beta$ is not bounded from above, one may wonder if the $(L,\beta ;D)$-branching diffusion is still well defined, in particular, whether the global (or even local) mass may blow up in finite time. To treat the case with $\beta$'s which are not upper bounded we will need  to consider more general branching diffusions. For a  `\emph{weighted  branching diffusion}', the particles do not necessarily carry unit mass. At each time $t>0$, the state of the process is  $\widehat X_{t}\in \widehat{\mathcal{M}}\left( D\right)$ where
\begin{equation*}
\widehat{\mathcal{M}}\left( D\right) =\left\{ \sum_{i=1}^{n}\gamma_i\delta _{x_{i}}:n\in
\mathbb{N}, \gamma_i>0\text{ and }x_{i}\in D\text{ for }i=1,...,n\right\}.
\end{equation*}
Next we need a definition. Let
\begin{equation*}
\lambda _{c}=\lambda _{c}(L+\beta ,D):=\inf \{\lambda \in \mathbb{R}\ :\
\exists u>0\ \text{satisfying}\ (L+\beta -\lambda )u=0\ \text{in}\ D\}
\end{equation*}
denote the \textit{generalized principal eigenvalue }for $L+\beta $
on $D$. By standard theory, $\lambda_c<\infty$ whenever $\beta$ is upper bounded and, for general $\beta$,  there exists an $h>0$ satisfying that $(L+\beta -\lambda )h=0$, whenever $\lambda_c<\infty$.

In the latter case, let us define  the $\widehat{\mathcal{M}}\left( D\right)$-valued process $W$ as follows. Each particle performs a $Y$ motion and carries weight $h(x)$ at $x\in D$, and furthermore, when the particle's clock rings, according to the spatially varying rate $\beta$, the particle splits into two offspring, which perform independent $Y$ motions and carry weights according to the function $h$, etc. Since $\beta$ is not upper bounded, we have to rule out finite time explosions. Fortunately, since $h$ is a `harmonic' function, it is a straightforward exercise to show that the total mass process $|W|$ is a supermartingale, and in particular, it is a.s. finite for all $t>0$.
Then, since $\widehat X$ is well defined, so is the $\mathcal{M}\left( D\right)$-valued process $X$ defined by
$\frac{\mathrm{d} X_t}{\mathrm{d} \widehat X_t}:=e^{\lambda_c t}h^{-1},$ (i.e. $X_t(B):=e^{\lambda_c t}\langle h^{-1}\mathbf{1}_{B},\widehat X_t\rangle,\ t\ge 0,\ B\subset D$ Borel).

Therefore, from now on, \emph{we relax the assumption that $\sup_D \beta<\infty$ and replace it with the much milder assumption $\lambda_c<\infty$}.

\subsection{Motivation}

This paper concerns growth of mass on compact domains of branching particle
diffusions. In doing so we address a gap in the literature dating back to
the mid-seventies when the study of growth of typed branching processes on
compact domains of the type space was popularized by Asmussen and Hering.
Also we complement a recent revival in this field which has appeared amongst
the superprocess community.

Before discussing main results, we shall introduce the topic in detail.

\begin{definition}[Local extinction]
\label{le} Fix $\mu \in \mathcal{M}\left( D\right)$. We say that $X$
exhibits local extinction under $P_{\mu }$ if for every Borel set $B\subset
\subset D$, there exists a random time $\tau _{B}$ such that
\begin{equation*}
P_{\mu }(\tau _{B}<\infty )=1\mbox{ and }P_{\mu }(X_{t}(B)=0\ for\ all\
t\geq \tau _{B})=1.
\end{equation*}
[Here $B\subset \subset D$ means that $B$ is bounded and its closure is a
subset of $D$.]
\end{definition}

Local extinction has been studied by \cite{P96}, \cite{EP} (for superprocesses) and
\cite{EK} (for branching diffusions). To
explain their results, recall that
we assume that the generalized principal eigenvalue for $L+\beta $
on $D$ is finite.
In fact, $\lambda _{c}\leq 0$ if and only if there exists
a function $h>0$ satisfying $(L+\beta )h=0$ on $D$ -- see Section
4.4 in \cite{P95}. Following the papers \cite{P96,EP} where similar issues were addressed for superprocesses, in
\cite{EK}  the following was shown.

\begin{theorem}[Local extinction versus local exponential growth]
\label{main-theorem} Let $\mathbf{0}\neq\mu \in \mathcal{M}\left( D\right) $.

\begin{itemize}
\item[(i)]  $X$ under $P_{\mu }$ exhibits local extinction if and only if
there exists a function $h>0$ satisfying $(L+\beta )h=0$ on $D$, that is, if
and only if $\lambda _{c}\leq 0$.

\item[(ii)]  When $\lambda _{c}>0$, for any $\lambda <\lambda _{c}$ and $%
\emptyset \neq B\subset \subset D$,
\begin{eqnarray*}
P_{\mu }(\limsup\nolimits_{t\uparrow \infty }e^{-\lambda t}X_{t}(B)=\infty
)>0, \\
P_{\mu }(\limsup\nolimits_{t\uparrow \infty }e^{-\lambda
_{c}t}X_{t}(B)<\infty )=1.
\end{eqnarray*}
\end{itemize}
In particular, local extinction/local exponential growth does not depend on the initial measure $\mathbf{0}\neq\mu \in \mathcal{M}\left( D\right) $.
\end{theorem}
(In \cite{EK} it is assumed that $\beta$ is upper bounded, whereas in \cite{EP} only the finiteness of $\lambda_c$ is assumed. The proofs of \cite{EK} go through for this latter case too.)
On closer inspection this last theorem says that when $\lambda _{c}\leq 0$
mass `escapes out of $B$' even though the entire process may survive with
positive probability. (If $Y\;$is conservative in $D$ for example then it
survives with probability one). Further, when $\lambda _{c}>0$ mass
accumulates on all nonempty bounded open domains and in such a way that with
positive probability this accumulation grows faster than any exponential
rate $\lambda <\lambda _{c}.$ On the other hand, mass will not grow faster
than at the exponential rate $\lambda _{c}.$ It is natural then to ask
whether in fact $\lambda _{c}$ gives an exact growth rate or not. That is to
say, for each $\emptyset \neq B\subset \subset D$ do the random measures $%
\{\exp \{-\lambda _{c}t\}X_{t}:t\geq 0\}$ converge in the vague topology
almost surely? Further, can one identify the limit? This is precisely the object of interest of a variety of previous studies for both branching diffusions and
superprocesses which we shall now review.

We note already here that the process in expectation is given by the \textit{%
linear kernel} corresponding to the operator $L+\beta $ on $D$. Therefore,
trusting in SLLN for branching processes, one should expect that the process
itself grows like the linear kernel too. On the other hand, it is easy to
see that the linear kernel does not in general scale precisely with $\exp
\{-\lambda _{c}t\}$ but rather with $f(t)\exp \{-\lambda _{c}t\}$, where $f$
grows to infinity as $t\rightarrow \infty $ and at the same time  is subexponential. (Take for
example $L=\Delta /2$ and $\beta >0$ constant on $\mathbb{R}^{d}$, then $%
f(t)=t^{d/2}$.) In fact the growth is pure exponential if and only if $%
L+\beta $ is product-critical (see later in this subsection). Proving SLLN seems to be
significantly harder in the general case involving the
subexponential term $f$.

In the late seventies Asmussen and Hering wrote a series of papers
concerning weak and strong laws of large numbers for a reasonably
general class of branching processes which included branching
diffusions. See \cite{AH76} and \cite{AH77}. In the context of the branching diffusions we consider here one can summarize briefly their achievements by saying that, when $D$ is \emph{bounded}, for a special class of
 operators $L+\beta $, the rescaled process $\{\exp
\{-\lambda _{c}t\}X_{t}:t\geq 0\}$ converges in the vague topology,
almost surely for branching diffusions.
Further, for the same class of $L+\beta $ when $D$ is unbounded they
proved that  there exists the limit \emph{in probability} of $\exp
\{-\lambda _{c}t\}X_{t} $ as $t\uparrow \infty$ (in the vague topology).
The class
of $L+\beta $ alluded to they called `positively regular'. The latter is
a subclass of the class $\mathcal{P}^*_p(D)$  (the class that we shall work with) given below.

A more detailed comparison with \cite{AH76, AH77} as well as the discussion on related results on superprocesses  is deferred to Section \ref{detailed}.

Before we give the definition of the basic classes of operators that we shall use, $\mathcal{P}_p(D)$ and $\mathcal{P}^*_p(D)$,  we need to recall certain concepts of
the so-called \textit{criticality theory} of second order operators.
The operator $L+\beta -\lambda _{c}$ is called \textit{critical} if
the associated space of positive harmonic functions is nonempty but
the operator does not possess a (minimal positive) Green's function.
In this case the space of positive harmonic functions is in fact
one-dimensional. Moreover, the space of positive harmonic functions
of the adjoint of $L+\beta -\lambda _{c}$ is also one dimensional.
\begin{assumption}\rm
Suppose we choose representatives of these two spaces to be $\phi $ and $%
\widetilde{\phi }$ respectively. Throughout the paper and without
further reference,  we will always assume that
$L+\beta -\lambda _{c}$ product-critical, and in this case we pick
$\phi $ and $\widetilde{\phi }$ with the normalization $\langle\phi,
\widetilde{\phi }\rangle=1$.
\end{assumption}
We now define the classes $\mathcal{P}_p\left( D\right)$ and $\mathcal{P}^*_p\left( D\right)$. Since we want to talk about spatial spread on a generic domain $D$, we fix, for the rest of the paper, an arbitrary family of domains  $\{D_t,\ t\ge 0\}$ with $D_t\subset\subset D,\ D_t\uparrow D$. (For $D=\mathbb R^d$, $D_t$ can be the $t$-ball, but we can take any other family with $D_t\subset\subset D,\ D_t\uparrow D$ too.)
\begin{definition}[$\mathcal{P}_p\left( D\right)$ and $\mathcal{P}^*_p\left( D\right)$]\rm
\label{class} For $p\geq 1$, we write $L+\beta\in\mathcal{P}_p\left( D\right) $
if
\begin{description}
\item[(i)]  $\lambda _{c}=\lambda _{c}(L+\beta ;D)>0,$
\item[(ii)] $\langle\phi^p, \widetilde{\phi }\rangle<\infty$, in which case we say that $L+\beta -\lambda _{c}$ is \emph{product $p$-critical}.
\end{description}

Let  $q(x,y,t)$ be transition density of $L+\beta$ and $
Q(x,y,t):=q(x,y,t)-e^{\lambda_c t}\widetilde{\phi}(y)\phi(x).$
We write $L+\beta\in\mathcal{P}^*_p\left( D\right)$ when the following additional conditions holds.
\begin{description}
\item[(iii)]
 For all $\emptyset\neq B\subset\subset D$ there exists a positive function $\zeta$ with $\zeta(t)\uparrow\infty$ as $t\uparrow\infty$ and such that
\[
|Q(x,y,\zeta(t))|\leq \alpha_t\widetilde{\phi}(y)\phi(x),\  x\in D_{t},\ y\in B,\ t\geq 0,
\]
where $\lim_{t\uparrow\infty}e^{-\lambda_c t}\alpha_t=0$.
\item[(iv)]
Furthermore,  for any given $x\in D$  there exists a  function $a:[0,\infty)\to [0,\infty)$ such that $\zeta(a_t)=\mathcal{O}(t)$ as $t\to\infty$ and  for all $\delta>0$,
$$P_{\delta_x}\left(\exists n_0, \forall n>n_0\ :\ \mathrm{supp}(X_{n\delta})\subset D_{a_{n\delta}}\right)=1.$$

\end{description}
\end{definition}
 Let $p(x,y,t)$ denote the transition density of the diffusion corresponding to the operator $(L+\beta-\lambda_c)^{\phi}$. Then $p(x,y,t)=e^{-\lambda_c t}\phi(y)\phi^{-1}(x)q(x,y,t)$, and thus, (iii) is equivalent to
\begin{description}
\item[(iii*)]With the same $\zeta$  as in (iii),

\[
\lim_{t\to\infty}\sup_{x\in D_{t}, y\in B}\left|\frac{p(x,y,\zeta(t))}{\phi\widetilde{\phi}(y)}-1\right|=0.
\]
\end{description}
Some heuristics to help {find} suitable $a$ and $\zeta$ will be discussed in Section \ref{Examples}.

\begin{remark}[Ergodicity]\rm
Note that criticality is invariant under $h$-transforms. Moreover,
an easy computation shows that $\phi$ and $\tilde\phi$ transforms
into $1$ and $\phi\tilde\phi$ respectively when turning from $\left(
L+\beta -\lambda_{c}\right) $ to the $h$-transformed ($h=\phi$)
operator  $(L+\beta -\lambda _{c})^{\phi }=L+a\phi ^{-1}\nabla
\phi \cdot \nabla .$ Therefore
product criticality is invariant under $h$-transforms too (this is
not the case with product $p$-criticality when $p>1$). Further, for
operators with no zeroth order term, it is equivalent to positive
recurrence (ergodicity) of the corresponding diffusion process. In
particular, $\left( L+\beta -\lambda_{c}\right) ^{\phi }$
corresponds to an ergodic diffusion process provided $\left( L+\beta
-\lambda_{c}\right) $ is product critical (see \cite{P95}, Section
4.9).$\hfill\diamond$
\end{remark}

\subsection{Main results}
With the following theorem we wish to address the issue of almost
sure convergence in the vague topology of $\{\exp \{-\lambda
_{c}t\}X_{t}:t\geq 0\}$ for branching diffusions with $L+\beta $
belonging  to $\mathcal{P}_p^*(D),\ p>1$ thus generalizing the results of Asmussen and Hering.

Note that since $L+\beta -\lambda _{c}$ is  critical, $\phi $ is the
unique (up to constant multiples) {\it invariant positive function}
for the linear semigroup corresponding to $L+\beta -\lambda _{c}$
(Theorem 4.8.6. in \cite{P95}). Let $\{S_t\}_{t\ge 0}$ denote the
semigroup corresponding to $L+\beta$. It is a standard fact (sometimes called `the one particle picture') that
\begin{equation}\label{opp}
S_t (g)(x)=E_{\delta _{x}}\langle g ,X_{t}\rangle
\end{equation}
 for all
nonnegative bounded measurable $g$'s. Even  though $\phi$ is not
necessarily bounded from above, $S_t (\phi )$ makes sense and (\ref{opp}) remains valid when $g$ is replaced by $\phi$, because
$\phi$ can be approximated with a monotone increasing sequence of
$g$'s and the finiteness of the limit is guaranteed precisely by the
invariance property of $\phi$. By the invariance of $\phi$,
$E_{\delta _{x}}e^{-\lambda _{c}t}\langle \phi ,X_{t}\rangle
=e^{-\lambda _{c}t}S_t(\phi) (x)=\phi \left( x\right) $, which is
 sufficient together with the branching
property to deduce that $W^{\phi }=\left\{ W_{t}^{\phi };\,t\geq 0\right\}
$
is a martingale where
\[
W_{t}^{\phi }=e^{-\lambda
_{c}t}
\langle \phi ,X_{t}\rangle
, \ t\geq 0.
\]
Indeed note that
\begin{equation*}
E_{\delta_{x} }\left( e^{-\lambda _{c}(t+s)}\langle \phi ,X_{t+s}\rangle |\mathcal{F%
}_{t}\right) =e^{-\lambda _{c}t}E_{X_{t}}\left( e^{-\lambda _{c}s}\langle
\phi ,X_{s}\rangle \right) =e^{-\lambda _{c}t}\langle \phi ,X_{t}\rangle .
\end{equation*}
Being a positive martingale, $P_{\delta_{x}}$-almost sure convergence is
guaranteed, and the  a.s. martingale limit $W_{\infty }^{\phi }:=\lim_{t\to\infty}W_{t}^{\phi }$ appears in  the following main conclusion.
\begin{theorem}[SLLN]\label{SLLN}Assume that  $L+\beta\in\mathcal{P}_p^*(D)$  for some
$p\in (1,2]$ and  $\langle\beta \phi^p,\widetilde{\phi}\rangle<\infty$.
 Then,
\begin{equation}
\lim\nolimits_{t%
\uparrow \infty }e^{-\lambda _{c}t}
\langle g,X_{t}\rangle
=\langle g,\widetilde{\phi }\rangle W_{\infty }^{\phi }, \ g\in C_{c}^{+}\left( D\right)
\label{strong.law}
\end{equation} holds $P_{\delta_x}-a.s.$ for all $x\in D$,
and $E_{\delta_x}\left(W_{\infty }^{\phi }\right)=1$.

Moreover, if $\sup_{D}\beta<\infty$ then the restriction $p\in (1,2]$ can be replaced by $p>1$.
\end{theorem}

We close this subsection with the Weak Law of Large Numbers. Here we change the class $\mathcal{P}_p^*(D)$ to the larger class $\mathcal{P}_p(D)$ and get convergence in probability instead of a.s. convergence. It is important to point out, however, that  the class $\mathcal{P}_p^*(D)$ is already quite large --- see Section \ref{Examples}, where we verify that key examples from the literature are in fact in $\mathcal{P}_p^*(D)$ and thus obey the SLLN.
\begin{theorem}[WLLN]\label{WLLN}
Suppose that $L+\beta\in\mathcal{P}_p(D)$ for some $p\in(1,2]$ and $\langle\beta\phi^p,\widetilde{\phi}\rangle<\infty$. Then for all $x\in D$,
(\ref{strong.law}) holds
in $P_{\delta_x}-$probability and $E_{\delta_x}(W^\phi_\infty)=1$. Moreover, if $\sup_D\beta<\infty$ then the restriction $p\in(1,2]$ can be replaced by $p>1$.
\end{theorem}
\subsection{Outline} The rest of this paper is organized as follows. In Section \ref{detailed} we embed our results into the literature, while in Section \ref{Examples} we discuss some key examples for the SLLN. The proofs are given in Section 4.
\section{Detailed comparison with some older results}\label{detailed}
The methods of Asmussen and Hering were based for the most part on classical
techniques of truncation and applications of the Borel-Cantelli Lemma. Using this method, they proved the convergence of $e^{-\lambda_c t}\langle X_t,g\rangle $ for all   $0\le g\in L^1(\widetilde\phi(x)\mathrm{d}x)$. It is also worth noting that the generic strength of their method extended to many
other types of branching processes; discrete time, discrete space and so on.

Interestingly, preceding all work of Asmussen and Hering is the
single article \cite{W} (later improved upon by \cite{B}). Watanabe demonstrates that when a suitable
Fourier analysis is available with respect to the operator $L+\beta
$, then by spectrally expanding any $g\in C_{c}^{+}\left( D\right)
,$\ the space of
nonnegative, continuous and compactly supported functions, one can show that $%
\{\langle g,X_{t}\rangle :t\geq 0\}$\ is almost surely asymptotically
equivalent to its mean. From this the classic Strong Law of Large Numbers
for dyadic branching Brownian motion in $\mathbb{R}^{d}$\ is recovered.\
Namely that when $L=\triangle /2$\ and $\beta $\ $>0$\ is a constant,
\begin{equation*}
\lim\nolimits_{t\uparrow \infty }t^{d/2}e^{-\beta t}X_{t}\left( B\right)
=\left( 2\pi \right) ^{d/2}\left| B\right| \times N_{\mu }
\end{equation*}
where $B$\ is any Borel set, $|B|$\ is its Lebesgue measure and $N_{\mu }$\
is a strictly positive random variable depending on the initial
configuration $\mu \in \mathcal{M}(\mathbb{R}^{d})$. The operator $1/2\Delta +\beta $
does not fall into the class $\mathcal{P}_1(D) .$ For an analogous result on supercritical super-Brownian motion see \cite{E}.

Let us discuss now how our assumptions relate to the assumptions imposed in the article \cite{AH76}.

In \cite{AH76} the  domain is bounded (and even one dimensional) when the Strong Law of Large Numbers is stated for branching diffusions; on general domains, only convergence in probability was obtained.
Furthermore, in \cite{AH76} the notion of {\it positively regular} operators was
introduced. In our context it first means  that
\begin{enumerate}
\item[\bf{(A)}] $\lambda_c>0$  (they call this property `supercriticality').
\item[\bf{(B)}] $\phi$ is bounded from above,
\item[\bf{(C)}] $\langle\tilde \phi, 1 \rangle <\infty.$
\end{enumerate}
Obviously, (B-C) is  stronger than the assumption $\langle \phi
,\widetilde \phi\rangle<\infty$ (product-criticality).

Secondly, $\{S_t\}_{t\ge 0}$, the semigroup corresponding to
$L+\beta$ (the so called `expectation semigroup') satisfies the
following condition. If $\eta$ is a nonnegative, bounded measurable
function on $\mathbb R^d$, then
\begin{equation*}\label{ah.unifconv}
\mathbf{(D)}\ \ \ \ \ \ \ S_t(\eta)(x) =\langle \eta , \widetilde
\phi \rangle\, \phi (x)\left[ e^{\lambda_c
t}+o\left(e^{\lambda_c t}\right)\right]\ \mathrm{as}\
t\uparrow\infty,\ \mathrm{uniformly\ in\ } \eta.
\end{equation*}
 Let $T_t$ be the semigroup defined by
$T_t(f):=S^{\phi}_t(f)=\frac{1}{\phi}S_t (\phi f),$ for all $0\le
f$ measurable with  $\phi f$ being bounded. Then $T_t$ correspond
to the $h$-transformed ($h=\phi$) operator $L_0^{\phi}$. Recall
that $L_0^{\phi}$ corresponds to a positive recurrent diffusion
process. Then, assuming that $\phi$ is bounded, it is easy to
check that the following condition would suffice for
(D) to hold:
\begin{equation}\label{fishy}
\lim_{t\uparrow \infty}\sup_{x\in D}\sup_{\| g\| \le 1 }\langle
g,\phi\widetilde \phi\rangle^{-1}\left|\,T_t (g)-\langle
g,\phi\widetilde \phi\rangle \,\right|=0,
\end{equation}
where $\|\cdot\|$ denotes $\sup$-norm. However this is not true in most cases on unbounded domains (or even on bounded domains with general unbounded coefficients) because of the requirement on the uniformity in $x$. (See our examples in Section \ref{Examples} --- neither of the examples on $\mathbb R^d$ satisfy (\ref{fishy}).)

Turning to superprocesses, there would seem to be considerably fewer
results of this kind in the literature (see the references \cite{Dawson,Dy,Et} for superprocesses in general). The most recent and general work in this area we are aware of are \cite{ET,EW,E}.

In \cite{ET} it was proved that  (in the vague topology)  $\{\exp
\{-\lambda _{c}t\}X_{t}:t\geq 0\}$ converges in law where $X$ is the
so called $(L,\beta,\alpha, \mathbb R^d )$-superprocess (with
$\alpha$ being the `activity parameter') satisfying that $L+\beta\in
\mathcal{P}\left(D\right)$ and  that $\alpha\phi$ is bounded from
above. (An additional requirement was that $\langle \phi,
\mu\rangle<\infty$ where $\mu=X_0$ is the deterministic starting
measure. ) The long and technical proof relied heavily on the theory
of dynamical systems applied to the Laplace transforms of $\{e^{-\lambda _{c}t}\phi X_{t},t\geq 0\}.$

In \cite{EW} the  convergence in law  was replaced by
convergence in probability. Furthermore, instead of  $\mathbb {R}^d$
 a general Euclidean domain $D\subseteq \mathbb
{R}^d$ was considered. The heavy analytic method of \cite{ET} was
replaced by a different, simpler and more probabilistic one. The
main tool was the introduction of a `weighted superprocess' obtained
by a `space-time $H$-transform'.
\section{Examples}\label{Examples}
In this section we give examples which satisfy all the assumptions we have, and thus, according to Theorem \ref{SLLN}, obey the SLLN. (Those  examples  do not fall into the setting in \cite{AH76, AH77}.)

Before we turn to the specific examples, we give some heuristics. Although these are not actually needed for understanding the examples, we feel that the reader `gets a more complete picture' by first reading them.
%
%
\begin{remark}[Expectation calculations and local vs. global growth rates]\rm
From (\ref{opp}), we have
\begin{align*}
E_x\langle \mathbf{1}_{\{\ \cdot\ \in dy\}},X_t \rangle
&= e^{\lambda_c t}\,\frac{\phi(x)}{\phi(y)}\,  p(t,x,y)
\end{align*}
and then, by ergodicity,
\[
e^{-\lambda_c t}\,E_x\langle 1,X_t \rangle = \phi(x)\,\int_{\mathbb{R}^d} \frac{p(t,x,y)}{\phi(y)}\,dy
\rightarrow \phi(x) \int_{\mathbb{R}^d} \widetilde\phi(y)\,dy,\ \text{as}\ t\to\infty.
\]
Hence, if $\langle\widetilde\phi, 1\rangle<\infty$, then the global population growth is the same as the local population growth,
whereas, if $\langle\widetilde\phi, 1\rangle<\infty$ the global growth rate exceeds the local growth rate.

\end{remark}
\begin{remark}[Heuristics for $a_t$ and $\zeta(t)$]\rm
One may wonder how  one can find a function $a$ as in Definition \ref{class}(iv).
In fact, this will often be straightforward to find.

Suppose that $Y$ is conservative and fix $x\in D$.
(If $Y$ is not conservative, then there is no function $a$ satisfying (iv). Indeed,  $D_t\uparrow$ implies that $\cup_{t\le T}D_t\subset D_{T}$, for $T>0$. Now, if  $Y_{T}=\Delta$ with positive probability for some $T>0$, then  the requirement $D_{a_{T}}\subset\subset D$ cannot hold.)
If, for example, we can pick a deterministic increasing function
$a$ such that, for all $\delta>0$,
$$\sum_{n=1}^{\infty}P_{\delta_x}(\mathrm{supp}(X_{n\delta})\not\subset D_{a_{n\delta}}) <\infty,$$
then Borel-Cantelli says that the function $a$ is an appropriately choice.
Since the probability one particle is present in a set is trivially dominated by the expected numbers in that set, it will be much easier to check that
$$\sum_{n=1}^{\infty} E_{\delta_x}\langle{\mathbf{1}_{D_{a_{n\delta}}^c}},X_{n\delta}\rangle <\infty.$$

If we can choose $a_t$ such that, for some $\epsilon>0$,
\[
\int_{|y|>a_t} \frac{p(t,x,y)}{\phi(y)}\,dy<e^{-(\lambda_c+\epsilon)t}
\]
then we will have satisfied
\[
\sum_{n=1}^{\infty} E_x\langle{\mathbf{1}_{D_{a_{n\delta}}^c}},X_{n\delta}\rangle <\infty.
\]

Heuristically,
\emph{if} the spine transition density $p(t,x,y)$ converges to its equilibrium $\phi(y)\widetilde\phi(y)$
sufficiently quickly even for very large $y$, we might hope to take
\[
a_t\approx\widetilde\phi^{-1}\left(e^{-\lambda_c t}\right)
\]

If the spine starts at a very large position,
since it is ergodic it will tend to move back toward the origin and
Ventcel-Friedlin  large deviation theory suggests that it will `closely' follow the
path of a deterministic particle with the same drift function.
We can use this to \emph{guess} for a suitable form for $\zeta(t)$.
At least heuristically, to find out how far away the spine particle may start in order that it both returns
to the vicinity of the origin and then ergodizes towards its invariant measure before large time $t$,
we can solve the deterministic differential equation
$$\dot f(t)=\mu(f(t))-a(f(t))\frac{\nabla\phi(f(t))}{\phi(f(t))}$$
when $L=\frac{1}{2}a(x)\Delta -\mu(x) \cdot\nabla$, and take $\zeta(t)$ a little larger than
$|f(t)|$.

Indeed, these heuristics appear to the correct form for both $a_t$ and $\zeta(t)$ in the examples considered below.

\end{remark}

\begin{example} [OU process with quadratic breeding rate]\rm
Let $\sigma, \mu, a, b>0$ and consider
\[
L:=\frac{1}{2}\sigma^2\Delta -\mu x\cdot\nabla\ \text{on}\ \mathbb{R}^d
\]
corresponding to an (inward) Ornstein-Uhlenbeck process and let $\beta(x):=b\,x^2+a.$
Since $L$ corresponds to a recurrent diffusion and $\beta$ is a smooth function with $\beta\ge 0$ and
$\beta\not\equiv 0$, it follows that $\lambda_c>0$ (see Chapter 4 in \cite{P95}).
The equilibrium distribution for $L$ is given by a normal distribution, 
\[
\pi(x)=
\left(\frac{\mu}{\pi\sigma^2}\right)^{d/2}
\exp\left\{-\frac{\mu}{\sigma^2}\ x^2\right\}.
\]
\emph{Suppose} that $\mu>\sigma\sqrt{2b}$.
Defining $\gamma^\pm:=\frac{1}{2\sigma^2} \left(\mu\pm\sqrt{\mu^2-2b\sigma^2}\right)$,
for the principle eigenvalue problem with $(L+\beta)\phi=\lambda_c\phi$ we can take
\[
\lambda_c:=\sigma^2\gamma^-+a,
\qquad
\phi(x):=c^-\, \exp\{\gamma^- x^2\}
\quad\text{and}\quad
\widetilde\phi(x)= c^+\, \exp\{-\gamma^+ x^2\}
\]
where
$c^-:=\left(1-(2b\sigma^2/\mu^2)\right)^{\frac{d}{8}}$,
$c^+:=c^-\left(\mu/(\pi\sigma^2)\right)^{\frac{d}{2}}$.
Note that $\widetilde\phi(x):=\phi(x)\pi(x)$ and $L+\beta$ is a self-adjoint operator with respect to $\pi$.

Some calculations using the `one-particle picture' (equation \ref{opp}) reveals that, in expectation,
the support of the process grows like $\sqrt{\lambda_c t/\gamma^+}$
and one can pick $a(t)=\sqrt{\lambda t/\gamma^+}$ for any $\lambda>\lambda_c$
and condition (iv) in Definition \ref{class} will hold.

The spine is also an (inward) Ornstein-Uhlenbeck process with parameter $\alpha:=\mu-2\gamma^-\sigma^2=\sqrt{\mu^2-2b\sigma^2}$
with
\[
(L+\beta-\lambda_c)^\phi=L+\sigma^2\frac{\nabla\phi}{\phi}\cdot\nabla = \frac{1}{2}\sigma^2\Delta -\alpha x\cdot\nabla\ \text{on}\ \mathbb{R}^d,
\]
and transition density
\[
p(t,x,y)
=\left(\frac{\alpha}{\pi\sigma^2\left(1-e^{-2(\alpha/\sigma^2) t}\right)}\right)^{d/2}
\exp\left[-\frac{\alpha\sum_{i=1}^d (y_i-x_ie^{-(\alpha/\sigma^2) t})^2}{\sigma^2(1-e^{-2(\alpha/\sigma^2) t})}\ \right].
\]
We see that the drift of the inward OU reduces the influence of any starting position
exponentially in time.
Indeed, one can take $\zeta(x)=(1+\epsilon)(\sigma^2/\alpha)\log x$ for any $\epsilon>0$
for condition ($iii^*$) in Definition \ref{class} to hold.
Finally, we trivially note that $\zeta(a(t))=\mathcal{O}(t)$ (in fact, only $\log t$ growth), hence all necessary conditions are satisfied for our strong law theorem to hold.

Note, a strong law for a generalization of this model can be found in \cite{H}
where the convergence is proved using a martingale expansion for continuous functions
$g\in L^2(\pi)$ (rather than compactly supported $g$).
Almost sure asymptotic growth rates (and a.s. support) for the same model are studied in \cite{GHH}.

This is certainly a non-trivial model and it highlights the strength of our general result.
In particular, a quadratic breeding rate is \emph{critical} in the sense that a BBM with breeding rate $\beta(x)=\text{const}\cdot x^p$ explodes in a finite time a.s. if and only if $p>2$, with explosion in the expected population size when $p=2$. When a branching inward OU process with quadratic breeding is considered here, a strong enough drift with $\mu>\sigma\sqrt{2b}$ can balance the high breeding, whereas any lower drift would lead to a dramatically different behavior.
$\hfill\diamond$
\end{example}
\begin{example}[Outward  OU process with constant breeding rate]\rm
Let $\sigma^2, \mu, b>0$ and consider
\[
L:=\frac{1}{2}\sigma^2\Delta + \mu x\cdot\nabla\ \text{on}\ \mathbb{R}^d
\]
corresponding to an `outward' Ornstein-Uhlenbeck process and let
$\beta(\cdot)\equiv b.$
As the spatial movements have no affect on the branching, the global population grows like $e^{\beta t}$ and this is achieved `naturally' with particles moving freely. This corresponds to $(L+\beta)\widetilde\phi=b \widetilde\phi$ with $\widetilde\phi\equiv 1$.
On the other hand, the principle eigenvalue is $\lambda_c=b-\mu<b$ with $\phi(x)=\text{const}\cdot \exp\{-(\mu/\sigma^2) x^2\}$, it being associated with the \emph{local}, as opposed to \emph{global}, growth rate.

After some similar expectation calculations to the inward OU in quadratic potential,
an upper bound on the process' spread is roughly the same as for an individual outward OU particle, that is,
we can take $a(t)=\exp\{(1+\delta) (\mu/\sigma^2) t\}$ for any $\delta>0$.

Despite the transient nature of the original motion, the spine is an inward OU process
\[
(L+\beta-\lambda_c)^\phi=L+\sigma^2\frac{\nabla\phi}{\phi}= \frac{1}{2}\sigma^2\Delta -\mu x\cdot\nabla\ \text{on}\ \mathbb{R}^d,
\]
with equilibrium $\phi\widetilde\phi(x)\propto  \exp\{-(\mu/\sigma^2) x^2\}$.
Intuitively, this is the motion that maximizes the \emph{local} growth rate at $\lambda_c$ (here its the original motion `conditioned to keep returning to the origin').
We can therefore take $\zeta(x)=(1+\epsilon)(\sigma^2/\mu)\log x$ for any $\epsilon>0$ and
hence still find that $\zeta(a_t)=(1+\epsilon)(1+\delta)t=\mathcal{O}(t)$.
All the conditions required for the strong law to hold are again satisfied.
$\hfill\diamond$
\end{example}

\begin{example}[BBM  with $\beta\in C_c^+(\mathbb R^d)$ and $\beta\not\equiv 0$ for $d=1,2$]\rm
Consider the $(\frac{1}{2}\Delta+\beta)$-branching diffusion where $\beta\in C_c^+(\mathbb R^d)$ and $\beta\not\equiv 0$ for $d=1,2$. Since Brownian motion is recurrent in dimension $d=1,2$, it follows that $\lambda_c>0$ and in fact, the operator $\frac{1}{2}\Delta+\beta-\lambda_c$ is product-critical and even product-$p$-critical for all $p>1$ (see Example 22 in \cite{ET}).

We now show how to find a $\zeta$ that satisfies ($iii^*$) in Definition \ref{class}. We do it for $d=1$, the $d=2$ case is similar.

Let $b>0$ be so large that supp$(\beta)\subset [-b,b]$ and let $M:=\max_{\mathbb R} \beta$. Recall that $p(t,x,y)$ denotes the (ergodic) kernel corresponding to $(\frac{1}{2}\Delta+\beta-\lambda_c)^{\phi}$. In this example $P$ will denote the corresponding probability. By comparison with the constant branching rate case, it is evident that $a_t:=\sqrt{2M}\cdot t$ is an appropriate choice. Therefore we have to find a $\zeta$ which satisfies that for any fixed ball $B$,
$$\lim_{t\to\infty}\sup_{|x|\le t}\left|\frac{p(x,B,\zeta(t))}{\int_B \phi\widetilde\phi(y)\, dy}-1\right|= 0$$
together with the condition that $\zeta(a_t)=\zeta (\sqrt{2M}\cdot t)=\mathcal{O} (t)$ as $n\to\infty$.

An easy computation (see again  Example 22 in \cite{ET}) shows that on $\mathbb R\setminus [-b,b]$, $$\left(\frac{1}{2}\Delta+\beta-\lambda_c\right)^{\phi}=\frac{1}{2}\Delta-\text{sgn} (x)\cdot \sqrt{2\lambda_c}\,\frac{\text{d}}{\text{d}x},$$
where sgn$(x):=x/|x|,\ x\neq 0$. Fix an $\epsilon$ and let $\tau_{\pm b}$ and $\tau_0$ denote the first hitting time (of a single Brownian particle) of $[-b,b]$ and of $0$, respectively.  We first show that as $t\to\infty$,
\begin{equation}\label{velocity}
\sup_{b<|x|\le t} P_x\left[\tau_{\pm b}>\frac {t(1+\epsilon)}{\sqrt{2\lambda_c}}\right]\to 0.
\end{equation}
Obviously, it is enough to show that for example $$\textbf{P}_t\left[\tau_0>\frac {t(1+\epsilon)}{\sqrt{2\lambda_c}}\right]\to 0$$
where $\textbf{P}$ corresponds to $\frac{1}{2}\Delta- \sqrt{2\lambda_c}\,\frac{\text{d}}{\text{d}x}$ on $[0,\infty)$. Indeed, if $\mathcal{W}$ denotes standard Brownian motion  starting at the origin with probability $Q$, then \begin{eqnarray*}\textbf{P}_t\left[\tau_0>\frac {t(1+\epsilon)}{\sqrt{2\lambda_c}}\right] \le \textbf{P}_t \left[Y_{\frac {t(1+\epsilon)}{\sqrt{2\lambda_c}}}>0\right]=Q\left[t-\sqrt{2\lambda_c}\frac {t(1+\epsilon)}{\sqrt{2\lambda_c}}+\mathcal{W}_{\frac {t(1+\epsilon)}{\sqrt{2\lambda_c}}}>0\right]\\=Q\left[\mathcal{W}_{\frac {t(1+\epsilon)}{\sqrt{2\lambda_c}}}>\epsilon t\right]\to 0\end{eqnarray*}
(the last term tends to zero by the SLLN for $\mathcal{W}$).

We now claim that $\zeta(t):=\frac {t(1+2\epsilon)}{\sqrt{2\lambda_c}}$ satisfies $$\lim_{t\to\infty}\sup_{|x|\le t}\left|\frac{p(x,B,\zeta(t))}{\int_B \phi\widetilde\phi(y)\, \mathrm{d}y}- 1\right|=0.$$
(The condition $\zeta(a_t)=\mathcal{O} (t)$ is obviously satisfied.) By the ergodicity of $p(t,x,y)$, it is sufficient to show that $\zeta$ satisfies $$\lim_{t\to\infty}\sup_{b<|x|\le t}\left|\frac{p(x,B,\zeta(t))}{\int_B \phi\widetilde\phi(y)\, \mathrm{d}y}- 1\right|=0.$$
Let, for example  $b<x<t$. By the strong Markov property at $\tau_{b}$ (the hitting time of $b$) and by (\ref{velocity}),
$$\frac{p(x,B,\zeta(t))}{\int_B \phi\widetilde\phi(y)\, dy}=\frac{p\left(b,B,\zeta(t)-\frac {t(1+\epsilon)}{\sqrt{2\lambda_c}}\right)}{\int_B \phi\widetilde\phi(y)\, dy}\,P\left[\tau_b\le\frac {t(1+\epsilon)}{\sqrt{2\lambda_c}} \right] + o(1),$$ uniformly in $b<x\le t$.

Finally, $$\lim_{t\to\infty}\frac{p(b,B,\zeta(t)-\frac {t(1+\epsilon)}{\sqrt{2\lambda_c}})}{\int_B \phi\widetilde\phi(y)\, dy}= 1$$ because $p(t,x,y)$ is an ergodic kernel and $$\lim_{t\to\infty}\left[\zeta(t)-\frac {t(1+\epsilon)}{\sqrt{2\lambda_c}}\right]=\lim_{t\to\infty}\frac {t\epsilon}{\sqrt{2\lambda_c}}=\infty,$$ completing the proof of our claim about $\zeta$.
$\hfill\diamond$
\end{example}
\begin{example}[Bounded domain]\rm
First note that when $D$ is bounded, an important subset of $\mathcal{P}_p(D),\ p>1$  is formed by the operators $L+\beta $ which are uniformly elliptic on $D$ with  bounded coefficients which are smooth up to the boundary of $D$ and with $\lambda _{c}>0$. That is, in this case $L+\beta -\lambda_c$ is critical (see \cite{P95}, Section 4.7), and since $\phi$ and $\widetilde\phi$ are Dirichlet eigenfunctions (are zero at the boundary of $D$), it is even product-$p$-critical for all $p>1$. Theorem \ref{WLLN} thus applies.

Although in this case $Y$ is not conservative in $D$, in fact even Theorem \ref{SLLN} will be applicable whenever $(iii^*)$ can be strengthened to the following uniform (in $x$) convergence on $D$:
\begin{equation}\label{lucky}
\lim_{t\to\infty}\sup_{x\in D, y\in B}\left|\frac{p(x,y,\zeta(t))}{\phi\widetilde{\phi}(y)}-1\right|=0.
\end{equation}
(Note that \cite{AH77} has a similar global uniformity assumption  --- see the paragraph after (\ref{fishy}).)
 Indeed, then the proof of Theorem \ref{SLLN} (which can be found later, in Section \ref{Proofs}) can  be simplified, because the function $a$  is not actually  needed: $D_{a_n}$ can be replaced by $D$ for all $n\ge 1$.

 As far as (\ref{lucky}) is concerned, it is often relatively easy to check. For example, assume that $d=1$ (the method can be extended for radially symmetric settings too) and so let $D=(r,s)$. Then the drift term of the spine is $b+a (\log \phi)'$. Now, if this is negative and bounded away from zero at $s-\epsilon<x<s$ and positive and bounded away from zero at $r<x<r+\epsilon$ with some $\epsilon\in(0,s-r)$, then (\ref{lucky}) can be verified by a method similar to the one in the previous example. The above condition on the drift is not hard to check in a concrete example (note that since $\phi$ satisfies the Dirichlet boundary condition, $\log \phi$ tends to $-\infty$ at the boundary).

If we relax the regularity assumptions on $L+\beta$ then for example $\phi$ is not necessarily upper bounded, and so we are leaving the family of operators handled in \cite{AH77} (see the four paragraphs preceding (\ref{fishy})); nevertheless our method still works as long as $L+\beta\in \mathcal{P}_p^*(D),\ p>1$ (for the SLLN) or $L+\beta\in \mathcal{P}_p(D),\ p>1$ (for the WLLN). 
 $\hfill\diamond$
\end{example}

\section{Proofs}\label{Proofs}

\subsection{A spine approach}
To establish the $L^{p}\left( P_{\delta_x }\right) $ convergence of
$W^{\phi }$ for $p>1$ we appeal to a, by now, standard
techniques that have been introduced to the literature by \cite{LPP}
and involves a change of measure inducing a `spine' decomposition.
Similar applications can be found in \cite{A,BK,EK,HHc}
to name but a few.
See for example \cite{Ev,Et} as well as the discussion in \cite{EK}
for yet further references.

It is important to point out that we will need the spine decomposition not only to establish the $L^{p}$-convergence mentioned above but also in the key lemma (Lemma \ref{keylemma}) in the proof of Theorem \ref{SLLN}. In both cases, we found the spine method to be indispensable and we were not able to replace it by other $L^p$ methods.

Before we can state our spine decomposition, we need to recall some facts
concerning changes of measures for diffusions and Poisson processes.

\textbf{Girsanov change of measure.} Suppose that $Y$ is adapted to some
filtration $\{\mathcal{G}_{t}:t\geq 0\}.$ Under the change of measure
\begin{equation}
\left. \frac{d\mathbb{P}_{x}^{\phi }}{d\mathbb{P}_{x}}\right| _{\mathcal{G}%
_{t}}=
\frac{\phi \left( Y_{t}\right) }{\phi \left( x\right) }%
e^{-\int_{0}^{t}\left(\lambda _{c}-\beta \left( Y_{s}\right)\right) ds}
\label{spinecofm}
\end{equation}
the process $(Y,\mathbb{P}_{x}^{\phi })$ corresponds to the $h$-transformed
$%
(h=\phi )$ generator $(L+\beta -\lambda _{c})^{\phi }=L+a\phi ^{-1}\nabla
\phi \cdot \nabla .$ Note now in particular that since $L+\beta\in
\mathcal{P%
}_1\left(D\right) $, it follows that $(Y,\mathbb{P}_{x}^{\phi })$ is
an ergodic diffusion with transition density $p(x,y,t)$ and an invariant
density
$\phi\tilde\phi$.

\textbf{Change of measure for Poisson processes.} Suppose that given a
non-negative continuous function $g(t),t\geq 0,$ the Poisson process $(n,%
\mathbb{L}^{g})$ where $n=\{\{\sigma _{i}:i=1,...,n_{t}\}:t\geq 0\}$ has
instantaneous rate $g(t).$ Further, assume that $n$ is adapted to $\{%
\mathcal{G}_{t}:t\geq 0\}.$ Then under the change of measure
\begin{equation*}
\left. \frac{d\mathbb{L}^{2g}}{d\mathbb{L}^{g}}\right| _{\mathcal{G}%
_{t}}=2^{n_{t}}\exp \left\{ -\int_{0}^{t}g\left( s\right) ds\right\}
\end{equation*}
the process $(n,\mathbb{L}^{2g})$ is also a Poisson process with
rate $2g.$ See  Chapter 3 in \cite{JS}.

\begin{theorem}[The spine construction]\label{spineconstruction}
Let $\{\mathcal{F}_{t}:t\geq 0\}$ be the natural filtration generated by
$X.$
Define the change of measure
\begin{equation*}
\left. \frac{d\widetilde{P}_{\delta _{x}}}{dP_{\delta _{x}}}\right| _{%
\mathcal{F}_{t}}=e^{-\lambda _{c}t}\frac{\langle \phi ,X_{t}\rangle }{\phi
(x)}=\frac{W_{t}^{\phi }}{\phi (x)}.
\end{equation*}
Then, under $\widetilde{P}_{\delta _{x}}$,
$X$ can be constructed as follows:
\begin{itemize}
\item a single particle, $Y=\{Y_t\}_{t\geq0}$, referred to as the
\emph{spine}, initially starts at $x$ and moves as a diffusion corresponding
to the $h$-transformed operator $L+a\phi ^{-1}\nabla \phi \cdot \nabla$;
\item the spine undergoes fission into two particles at an
\emph{accelerated} rate $2\beta(Y_t)$ at time $t$,
one of which is selected uniformly at random to continue the spine motion
$Y$;
\item the remaining child gives rise to an independent copy
of a $P$-branching diffusion started at its space-time point of creation.
\end{itemize}
\end{theorem}
An similar construction for BBM was established in Chauvin and Rouault
\cite{CR}.
See  \cite{EK} (Theorem 5) or \cite{HHa}, for example,  on how to prove it.

\begin{remark}[\textbf{The Spine decomposition.}]\rm
\label{spineremark}
Theorem \ref{spineconstruction} says that $(X,\widetilde{P}_{\delta _{x}})$
has the same law as a process constructed in the following way.
A $(Y,\mathbb{P}_{x}^{\phi })$-diffusion is initiated along which
$(L,\beta;D)$-branching processes immigrate at space-time points
$\{(Y_{\sigma _{i}},\sigma
_{i}):i\geq 1\}$ where, given $Y,\ n=\{\{\sigma_{i}:i=1,...,n_{t}\}:t\geq
0\}$ is a Poisson process with law $\mathbb{L}^{2\beta (Y)}$.
It will often be \emph{very} useful to think of $(X,\widetilde{P}_{\delta
_{x}})$
as being constructed in this richer way and it will be convenient to define
the natural
filtration of the spine and the birth process along the spine as
$\mathcal{G}_t:=\sigma(Y_s,n_s:s\leq t)$.
Note that using the `spine construction' of $(X,\widetilde{P})$ as discussed
in Remark \ref{spineremark}, we can write
\begin{equation*}
W^\phi_t = e^{-\lambda_c t}\phi(Y_t) + \sum_{i=1}^{n_t}
e^{-\lambda_c  \sigma_i} W_i,
\end{equation*}
where, conditional on the spine filtration $\mathcal{G}_t$,
$W_i$ is an independent copy of the martingale $W_t^\phi$ started from
position $Y_{\sigma_i}$ and
run for time $t-\sigma_i$ where $\sigma_i$ is the $i^{th}$ fission time
along the spine for $i=1,\dots,n_t$.
Remembering that particles off the spine behave the same as if under the
original measure $P$
and that the martingale property gives $P_{\delta_x}(W_t^\phi)=\phi(x)$, we
then have the so called
\emph{`spine decomposition'}:
\begin{equation}
\widetilde{E}\left(W^\phi_t\Big|{\mathcal{G}_t}\right)e^{-\lambda_c t}
\phi(Y_t)
+ \sum_{i=1}^{n_t} e^{-\lambda_c \sigma_i} \phi(Y_{\sigma_i})
\label{spinedecomp}
\end{equation}
\end{remark}

\textbf{The $L^p$-convergence of the martingale.} The a.s. convergence of
$W^{\phi }$ can be complemented with the following result.
\begin{lemma}\label{mgcgce}Assume that $L+\beta$ belongs to
$\mathcal{P}_p(D)$ and that $\langle\beta
\phi^p,\widetilde{\phi}\rangle<\infty$ for some $p\in (1,2]$.
Then, for $x\in D$, $W^{\phi }$ is an $L^{p}\left(
P_{{\delta_x}}\right)$-convergent martingale. Moreover, if $\sup_D
\beta<\infty$, then the same conclusion holds if we only assume that $p>1$.
\end{lemma}
\begin{proof}
Pick  $p$ so that  $q=p-1\in(0,1]$,
$\langle\phi^p,\widetilde{\phi}\rangle<\infty$ and $\langle\beta
\phi^p,\widetilde{\phi}\rangle<\infty$.
(If $K:=\sup_D \beta<\infty$ and we only assume that $p>1$, then $\langle
\phi^p,\widetilde{\phi}\rangle=\langle
\phi^{p-1},\phi\widetilde{\phi}\rangle<\infty$ implies $\langle
\phi^r,\widetilde{\phi}\rangle=\langle
\phi^{r-1},\phi\widetilde{\phi}\rangle<\infty$ and $\langle\beta
\phi^r,\widetilde{\phi}\rangle\le K\langle
\phi^{r-1},\phi\widetilde{\phi}\rangle<\infty$ for all $r\in(0,p)$, and so
we can assume that in fact $p\in(1,2].$)

We adopt an approach similar  to the one in \cite{HHa,HHb}.
position $Y_{\sigma_i}$ and
along the spine for $i=1,\dots,n_t$.
original measure $P$
then have the so called
Using the conditional form of Jensen's inequality, the spine decomposition
\eqref{spinedecomp}
and that
$(u+v)^q\leq u^q+v^q$ for $u,v>0$ when $q\in(0,1)$, we find
\begin{eqnarray*}
\phi(x)^{-1} E_{\delta_x}\left[(W^\phi_t)^p\right]&
=&\widetilde{E}_{\delta_x}\left[(W^\phi_t)^q\right]
=\widetilde{E}_{\delta_x}\left\{\widetilde{E}\left[(W^\phi_t)^q\Big|{\mathcal{G}_t}
\right]\right\}\\&\leq&
\widetilde{E}_{\delta_x}\left\{\left[\widetilde{E}\left(W^\phi_t\Big|{\mathcal{G}_t}
\right)\right]^q\right\}
\\&\leq&  \mathbb{E}_{x}^{\phi }\mathbb{L}^{2\beta (Y)}\left(e^{-\lambda_c
qt}\phi(Y_t)^q +\sum_{i=1}^{n_t} e^{-\lambda_c q \sigma_i}
\phi(Y_{\sigma_i})^q
\right)\\
&=& e^{-\lambda_c qt}\mathbb{E}_{x}^{\phi }[\phi(Y_t)^q ]+
\mathbb{E}_{x}^{\phi }\left[\int_0^t e^{-\lambda_c q s}
2\beta(Y_s)\phi(Y_{s})^q{\rm d}s
\right].
\end{eqnarray*}
Call the two expressions on the right hand side the \emph{spine term},
$A(x,t)$, and
the \emph{sum term}, $B(x,t)$, respectively.
Since $Y$ has generator $L+\,a\phi^{-1}(\nabla\phi)\!\cdot \!\nabla$ and
$\langle\phi,\widetilde{\phi}\rangle=1$, $Y$ is ergodic
and $\mathbb{E}_{x}^{\phi }(f(Y_t))\rightarrow\langle f
\phi,\widetilde{\phi}\rangle$.
Then $$\lim_{t\uparrow\infty}e^{\lambda_c
qt}A(t,x)=\lim_{t\uparrow\infty}\mathbb{E}_{x}^{\phi }(\phi(Y_t)^q )
=\langle\phi^p,\widetilde{\phi}\rangle<\infty$$
for all $x\in D$.
For the sum term note that
$\lim_{s\uparrow\infty} \mathbb{E}_{x}^{\phi
}\left(\beta(Y_s)\phi(Y_{s})^q\right)
=\langle\beta \phi^p,\widetilde{\phi}\rangle<\infty$ and so
$\lim_{t\uparrow\infty} B(t,x)<\infty$ for all $x\in D$.
By Doob's inequality, $W^\phi$ is therefore an ${L}^p$-convergent
(uniformly integrable) martingale, as required.
\end{proof}


\subsection{Proof of Theorem \ref{SLLN} along lattice times}
The statement that $E_{\delta_x}(W_{\infty}^{\phi})=1$ follows from Lemma \ref{mgcgce}.
The rest of the proof will be based on the following key lemma.
\begin{lemma}\label{keylemma}
Fix $\delta>0$ and let $B\subset\subset D$. Define
\[
 U_t = e^{-\lambda_c t}
 \langle \phi |_B, X_t\rangle,
\]
where $\phi |_B(x)=\phi(x)\mathbf{1}_{(x\in B)}$.
Then for any non-decreasing sequence  $\{m_n\}_{n\ge 1}$,
\[
\lim_{n\uparrow\infty} |U_{(m_n+n)\delta} - E(U_{(m_n+n)\delta}|\mathcal{F}_{n\delta})|=0,\ P_{\delta_x}-a.s.
\]
\end{lemma}

\begin{proof}We will suppress the dependence in $n$ in our notation below and simply write $m$ instead $m_n$. Suppose that $\{X_i :  i=1,...,N_{n\delta}\}$ describes the
configuration of particles at time $n\delta$. Note that we may always write
\begin{equation}
U_{(m+n)\delta} = \sum_{i=1}^{N_{n\delta}}e^{-n\delta\lambda_c}
U^{(i)}_{m\delta}
\label{always write}
\end{equation}
where given $\mathcal{F}_{n\delta}$, the collection $\{U^{(i)}_{m\delta}: i = 1,..., N_{n\delta}\}$ are mutually independent and equal in distribution to $U_{m\delta}$ under $P_{\delta_{X_i}}$ respectively.

By the Borel-Cantelli lemma, it is sufficient to prove that for $x\in D$ and for all $\epsilon>0$,
\[
\sum_{n\geq 1}P_{\delta_x}\left(\left| U_{(m+n)\delta} -
E(U_{(m+n)\delta}|\mathcal{F}_{n\delta})
\right|>\epsilon\right)<\infty.
\]
To this end we first note that,
\[
P_{\delta_x}\left(\left| U_{(m+n)\delta}
- E(U_{(m+n)\delta}|\mathcal{F}_{n\delta}) \right|>\epsilon\right)
\leq
\frac{1}{\epsilon^p} E_{\delta_x}\left(\left|U_{(m+n)\delta}
- E(U_{(m+n)\delta}|\mathcal{F}_{n\delta})  \right|^p\right).
\]
Now recall the following very useful result,
for example see \cite{B} or \cite{CH}:
if $p\in(1,2)$ and $X_i$ are independent random variables with $E(X_i)=0$
(or they are martingale differences), then
\begin{equation}
E\left|\sum_{i=1}^n X_i\right|^p
\leq 2^p \, \sum_{i=1}^n E\left|X_i\right|^p.
\label{basicLpineq}
\end{equation}
Jensen's inequality also implies that for each $n\geq 1$,
$|\sum_{i=1}^n u_i |^p\leq n^{p-1} \sum_{i=1}^n (|u_i|^p)$ and,
in particular,  $|u + v|^p \leq 2^{p-1}(|u|^p + |v|^p)$.

Note that
\[
U_{s+t}-E(U_{s+t}|\mathcal{F}_t)
\sum_{i=1}^{N_t} e^{-\lambda_c t}
\left( U_{s}^{(i)}-E(U_{s}^{(i)}|\mathcal{F}_t) \right)
\]
where conditional on $\mathcal{F}_t$,
$Z_i:=U_{s}^{(i)}-E(U_{s}^{(i)}|\mathcal{F}_t)$ are
independent with $E(Z_i)=0$.
Thus, by \eqref{basicLpineq} and Jensen,
\begin{eqnarray*}
&&
E \left(\left|U_{s+t} - E(U_{s+t}|\mathcal{F}_{t})  \right|^p \,|\, \mathcal{F}_t\right) \\
&&
\leq 2^p \, e^{-p \lambda_c t}
\sum_{i=1}^{N_{t}}
E\left( |U^{(i)}_{s} - E(U_{s}^{(i)}|\mathcal{F}_t) |^p  \,\Big| \, \mathcal{F}_{t}\right) \\
&&
\leq 2^p \, e^{-p \lambda_c t}
\sum_{i=1}^{N_{t}}
E \left( 2^{p-1}\left( |U^{(i)}_{s}|^p
+ | E(U_{s}^{(i)}|\mathcal{F}_t) |^p \right) \,\Big| \, \mathcal{F}_{t}\right)  \\
&&
\leq 2^p\, e^{-p \lambda_c t}
\sum_{i=1}^{N_{t}}
2^{p-1} E \left(  |U^{(i)}_{s}|^p
+  E(|U_{s}^{(i)}|^{p}|\mathcal{F}_t)  \,\Big| \, \mathcal{F}_{t}\right)  \\
&&
\leq 2^{2p}\, e^{-p \lambda_c t}
\sum_{i=1}^{N_{t}}
E\left( |U^{(i)}_{s}|^p \Big|\mathcal{F}_t\right)
\end{eqnarray*}

Then, as a consequence of the previous estimate, we have that
\begin{eqnarray}
&&\sum_{n\geq 1}E_{\delta_x}\left(\left|U_{(m+n)\delta} -
E(U_{(m+n)\delta}|\mathcal{F}_{n\delta})  \right|^p\right)\notag\\
&&\leq 2^{2p}   \sum_{n\geq 1}e^{-\lambda_c n\delta p}E_{\delta_x}\left(\sum_{i=1}^{N_{n\delta}}
E_{\delta_{X_i}}[(U_{m\delta})^p]
\right)
\label{borrow later}
\end{eqnarray}
Recalling the definition of the terms $A(x,t)$ and $B(x,t)$ from the proof of
Theorem \ref{mgcgce}(ii) and trivially noting that $U_t\leq W^\phi_t$, we have
\begin{eqnarray*}
&&\sum_{n\geq 1}E_{\delta_x}\left(\left|U_{(m+n)\delta} -
E(U_{(m+n)\delta}|\mathcal{F}_{n\delta})  \right|^p\right)\\
&&\leq
2^{2p}   \sum_{n\geq 1}e^{-\lambda_c n\delta p}
E_{\delta_x}\left( \sum_{i=1}^{N_{n\delta}}
E_{\delta_{X_i}}[(W^\phi_{m\delta})^p]
\right) \\
&&\leq
 2^{2p}\sum_{n\geq 1}
E_{\delta_x}\left(
\sum_{i=1}^{N_{n\delta}} e^{-p \lambda_c n\delta}
\phi(X_i)(A(X_i, m\delta) +B(X_i,m\delta))
\right) \\
&&=2^{2p}\sum_{n\geq 1}
\phi(x)e^{-q \lambda_c \delta n}\mathbb{E}_x^\phi\left(
A(Y_{n\delta},m\delta) +B(Y_{n\delta},m\delta) \right)
\end{eqnarray*}
where we have used the `one-particle picture' (equation (\ref{opp})) and the spine change of measure
at \eqref{spinecofm}.
Since the spine $Y$ is Markovian and ergodic under $\mathbb{P}_x^{\phi}$, we know that
\[
\mathbb{E}_x^\phi\left[A(Y_{n\delta},m\delta)\right] = e^{-\lambda_c q m\delta} \mathbb{E}_x^\phi (\phi(Y_{(m+n)\delta})^q).
\]Denoting $m_{\infty}:=\lim_{n\to\infty}m_n$, the latter converges to $e^{-q\lambda_c m_{\infty}\delta}\langle \phi^p,\widetilde{\phi}\rangle$ (which will be zero if $m_{\infty} =\infty$) as $n\uparrow\infty$.
Recall the assumption that $\langle \beta \phi^p,\widetilde{\phi}\rangle<\infty$.  Similarly as before, we have that
\[
\mathbb{E}_x^\phi\left[B(Y_{n\delta},m\delta)\right] = 2\int_0^{m\delta}e^{-\lambda_c qs}
 \mathbb{E}_x^\phi (\beta(Y_{s+n\delta})\phi(Y_{s+n\delta})^q){\rm d}s
\]
which has a finite limit equal to $2\int_0^{m_{\infty}\delta} e^{-\lambda_c s} \langle \beta \phi^p,\widetilde{\phi}\rangle {\rm d}s$ as $n\uparrow\infty$.
These facts are enough to conclude that the last sum remains
finite to complete the Borel-Cantelli argument.
\end{proof}

We now complete the proof of Theorem \ref{SLLN} along lattice times.
Assume that $L+\beta\in\mathcal{P}^*_p$ for some
$p>1$. Recall now that $I(B):=\int_B\phi(y)\widetilde{\phi}(y){\rm
d}(y)<1$ and  note that, similarly to (\ref{always write}),
\begin{eqnarray*}
&&E(U_{t+s}|\mathcal{F}_t)\\ &&=\sum_{i=1}^{N_t} e^{-\lambda_c
t}\phi(X_i)p(X_i,B,s)
=\sum_{i=1}^{N_t}e^{-\lambda_c (t+s)} \int_B\phi(y)q(X_i,y,s){\rm d}(y)\\
&&=\sum_{i=1}^{N_t}e^{-\lambda_c t} \phi(X_i)\,I(B)
+\sum_{i=1}^{N_t}e^{-\lambda_c (t+s)} \int_B\phi(y)Q(X_i,y,s){\rm d}(y)\\
&&=I(B) W^\phi_t +\sum_{i=1}^{N_t}e^{-\lambda_c (t+s)}
\int_B\phi(y)Q(X_i,y,s){\rm d}(y)=:I(B) W^\phi_t + \Theta(t,s).
\end{eqnarray*}
Let us replace now $t$ by $n\delta$ and $s$ by $m_n\delta$, where $$m_n:=\zeta(a_{n\delta})/\delta,$$ and $\zeta$ is the function appearing in the definition of $\mathcal{P}_p^*$. (Although we do not need it yet, we note that, according to (iv) in Definition \ref{class}, one has $m_n\le Kn$ as $n\to\infty$, where $K>0$ does not depend on $\delta$.) Then $$E(U_{(n+m_n)\delta}|\mathcal{F}_{n\delta})=I(B)W^\phi_{n\delta} + \Theta(n\delta,m_n\delta)=:I(B)W^\phi_{n\delta} +\Upsilon(n).$$
Define the event $$A_n:=\{\text{supp}(X_{n\delta})\not\subset D_{a_{n\delta}}\}.$$ Because of the choice of $m_n$ and since $I(B)<1$, we have
\begin{eqnarray*}
&&|\Upsilon(n)|\leq\\&& \sum_{i=1}^{N_{n\delta}}e^{-\lambda_c n\delta} \phi(X_i)e^{-\lambda_c m_n\delta}
\alpha_{m_n\delta}+|\Upsilon(n)| \mathbf{1}_{A_{n}}=
e^{-\lambda_c m_n\delta}\alpha_{m_n\delta}\, W^\phi_{n\delta}+|\Upsilon(n)| \mathbf{1}_{A_{n}}.
\end{eqnarray*}
Since, according to Definition \ref{class}(iv), $\lim_{n\to\infty}\mathbf{1}_{A_{n}}= 0,\ P-$a.s., therefore
$$\limsup_{n\uparrow\infty}|\Upsilon(n)|\le\lim_{n\uparrow\infty}
 e^{-\lambda_c m_n\delta}\alpha_{m_n\delta}\, W^\phi_{n\delta}=0\ \ \ P_{\delta_x}-a.s.,$$ and so
\begin{equation}\label{close}
\lim_{n\uparrow\infty}\left|E_{\delta_x}(U_{(n+m_n)\delta}|\mathcal{F}_{n\delta})
-\langle\phi |_B,\widetilde{\phi}dx \rangle W^\phi_\infty
\right|=0\ \ \  P_{\delta_x}-a.s.
\end{equation}
Since $\mathrm{Span}\{\phi |_B,\ B\subset\subset D\}$ is
dense in $C^+_c$, the result  for lattice times follows  by standard arguments along
with Lemma \ref{keylemma}. $\hfill\square$
\subsection{Replacing lattice times with continuous time}
 The following lemma is  enough to conclude the convergence in Theorem \ref{SLLN} (see the remark after the lemma).  It upgrades convergence along lattice times to the full sequence of times and is based on the idea to be found in Lemma 8 of \cite{AH76}.
\begin{lemma}\label{upgrade}
Suppose that for some $p>1$,  $\langle \phi^p, \widetilde{\phi}  \rangle<\infty$ and for all $\delta>0$ it is true that  for all $g\in C_c^+(D)$ and $x\in D$
\[
\lim_{n%
\uparrow \infty }e^{-\lambda _{c}n\delta}
\langle g,X_{n\delta}\rangle
=\langle g,\widetilde{\phi }\rangle W_{\infty }^{\phi }\
\qquad P_{\delta_x}-a.s.,
\]
then the same result holds  when $n\delta$ is replaced by $t$ and $\lim\nolimits_{n\uparrow \infty }$ by $\lim\nolimits_{t\uparrow \infty }$.
\end{lemma}
\begin{remark} \rm   Recall that we assumed that $\zeta(a_t)=\mathcal{O}(t)$ as $t\to\infty$, and so referring to the previous subsection, $m_n=\zeta(a_{n\delta})/\delta\le Kn$ with some $K>0$ which does not depend on $\delta$.
 In fact, by possibly further increasing the function $a$, we can actually take $\zeta(a_{t})= Kt$ and $m_n=Kn$. Then,  from the previous subsection we already know that \[
\lim\nolimits_{n%
\uparrow \infty }e^{-\lambda _{c}({K}+1)n\delta}
\langle g,X_{({K}+1)n\delta}\rangle
=\langle g,\widetilde{\phi }\rangle W_{\infty }^{\phi }\
\qquad P_{\delta_x}-a.s.
\]
Thus the assumption in Lemma \ref{upgrade} is indeed satisfied (write $\delta':=\delta({K}+1)$). $\hfill\diamond$
\end{remark}
\begin{proof}
First suppose that $B\subset\subset D$ and for each $x\in D$ and $\epsilon>0$, define
\[
B^\epsilon(x) = \{y\in B : \phi(y)>(1+\epsilon)^{-1} \phi(x)\}.
\]
Note in particular that $x\in B^\epsilon(x)$ if and only if $x\in B$. Next define for each $\delta>0$
\[
\Xi^{\delta, \epsilon}_B(x) = \mathbf{1}_{\{\text{supp} (X_t)\subset B^\epsilon(x) \text{ for all }t\in[0,\delta]\}}
\]
and let $\xi^{\delta, \epsilon}_B(x) =E_{\delta_x}(\Xi^{\delta, \epsilon}_B(x) )$. An important feature of the latter quantity in the forthcoming proof is that $\xi^{\delta, \epsilon}_B(x)\rightarrow \mathbf{1}_{B}(x)$ as $\delta\downarrow 0$.
With this notation we now note the crucial estimate
\[
e^{-\lambda_c t}\langle \phi|_B, X_t \rangle \geq \frac{e^{-\delta}}{(1+\epsilon)} \sum_{i=1}^{N_{n\delta}}   e^{-\lambda_c n\delta}\phi(X_i)\Xi^{\delta,\epsilon}_B(X_i).
\]
Note that the sum on the right hand side is of the form (\ref{always write}) where now $U_{(m+n)\delta}$ is played by the right hand side above and $U^{(i)}_{m\delta}$ is played by the role of $\phi(X_i)\Xi^{\delta,\epsilon}_B(X_i).$ Similar  $L^p$ estimates to those found in Lemma \ref{keylemma} show us that the estimate (\ref{borrow later}) is still valid in the setting here and hence
\begin{eqnarray*}
&&\sum_{n\geq 1}E_{\delta_x}\left(\left|U_{(m+n)\delta} -
E(U_{(m+n)\delta}|\mathcal{F}_{n\delta})  \right|^p\right)\notag\\
&&\leq
2^{2p}   \sum_{n\geq 1}e^{-\lambda_c n\delta p}E_{\delta_x}\left(\sum_{i=1}^{N_{n\delta}}
\phi(X_i)^p\xi^{\delta,\epsilon}_B(X_i)
\right).
\end{eqnarray*}
However, with $q=p-1$, the righthand side can  again be upper estimated by
\[
2^{2p}   \sum_{n\geq 1}e^{-\lambda_c n\delta p}E_{\delta_x}
\langle \phi^p, X_{n\delta}\rangle
=2^{2p}   \sum_{n\geq 1}e^{-\lambda_c n\delta q}\mathbb{E}^\phi_x
(\phi(Y_{n\delta})^q)<\infty
\]
where the equality follows by equation (\ref{opp}), and the fact that the final sum is finite, follows by the ergodicity of $\mathbb{P}^\phi_x$ and the assumption that $\langle\phi^p,\widetilde{\phi}\rangle<\infty$.

We may now  appeal to the Borel-Cantelli Lemma to deduce that
\[
\lim_{n\uparrow\infty}\left|\sum_{i=1}^{N_{n\delta}}   e^{-\lambda_c n\delta}\phi(X_i)\Xi^{\delta,\epsilon}_B(X_i)
-
e^{-\lambda_c n\delta} \langle  \phi\xi^{\delta,\epsilon}_B, X_{n\delta}\rangle
\right|=0
\]
$P_{\delta_x}$-almost surely and hence, using the fact that the Strong Law of Large Numbers has been proved already for $n\delta$-sequences,
\[
\liminf_{t\uparrow\infty}e^{-\lambda_c t}\langle \phi|_B, X_t\rangle \geq \frac{e^{-\delta}}{(1+\epsilon)} \langle\phi \xi^{\delta,\epsilon}_B, \widetilde{\phi}\rangle W^\phi_\infty.
\]
Taking $\delta\downarrow 0$ reveals that $\langle\phi \xi^{\delta,\epsilon}_B, \widetilde{\phi}\rangle\rightarrow \langle\phi |_B, \widetilde{\phi}\rangle$ in the lower estimate above, and hence subsequently taking $\epsilon\downarrow 0$ gives us
\[
\liminf_{t\uparrow\infty}e^{-\lambda_c t}\langle \phi|_B, X_t\rangle \geq \langle\phi |_B, \widetilde{\phi}\rangle W^\phi_\infty.
\]

Recall that this estimate was computed for the case that $B\subset\subset D$. Suppose now that $B\subseteq D$ (not necessarily bounded). Then there exists an increasing sequence of compactly embedded domains in $B$, say $\{B_n : n\geq 1\}$, such that $\bigcup_{n\geq 1}B_n  =B$. We may then note that for each $n\geq 1$
\[
\liminf_{t\uparrow\infty}e^{-\lambda_c t}\langle \phi|_B, X_t\rangle \geq
\liminf_{t\uparrow\infty}e^{-\lambda_c t}\langle \phi|_{B_n}, X_t\rangle \geq
\langle\phi |_{B_n}, \widetilde{\phi}\rangle W^\phi_\infty,
\]
and hence, as $n$ is arbitrary,
\[
\liminf_{t\uparrow\infty}e^{-\lambda_c t}\langle \phi|_B, X_t\rangle \geq
\langle\phi |_{B}, \widetilde{\phi}\rangle W^\phi_\infty,\ \ P_{\delta_x}-a.s.
\]
Now that we have a tight lower estimate for the liminf for arbitrary Borel $B\subseteq D$, we shall look at the limsup, also for arbitrary Borel $B\subseteq D$. Using the normalization $\langle\phi, \widetilde{\phi}\rangle=1$, one has
\[
\limsup_{t\uparrow\infty}e^{-\lambda_c t}\langle \phi|_B, X_t\rangle =
W^\phi_\infty - \liminf_{t\uparrow\infty}e^{-\lambda_c t}\langle \phi|_{D\backslash B}, X_t\rangle \leq
\langle\phi |_{B}, \widetilde{\phi}\rangle W^\phi_\infty,\ \ P_{\delta_x}-a.s.
\]
This, together with the  liminf result, yields
\[
\lim_{t\uparrow\infty}e^{-\lambda_c t}\langle \phi|_B, X_t\rangle =
\langle\phi |_{B}, \widetilde{\phi}\rangle W^\phi_\infty,\ \ P_{\delta_x}-a.s.
\]

Then, just like for lattice times, a straightforward measure theoretic consideration shows that $\phi |_B$ can be replaced by an arbitrary test function $g\in C_c^+(D)$, completing the proof.
\end{proof}
\subsection{Proof of Theorem \ref{WLLN}}
\begin{proof}The last part of the theorem is merely a consequence of Lemma \ref{mgcgce}. For any $g\in C_c^+(D)$ define for each $x\in D$ the function $h_s(x) = \mathbb{E}_x^\phi(g(\xi_s))$, and note that, uniformly in $x$ and $s$, the function $h_s(x)$ is bounded. Now define $U_t [g]= e^{-\lambda_c t}\langle g\phi, X_t\rangle$ and observe that, just as in Theorem \ref{SLLN}, one has
\[
U_{t+s}[g] = \sum_{i=1}^{N_t}e^{-\lambda_c t} U^{(i)}_s[g],
\]
where by  (\ref{opp}), $$E(U^{(i)}_s[g]|\mathcal{F}_t) = \phi(X_i(t))h_s(X_i(t)).$$
Next, note from the Markov property at $t$ and the proof \footnote{Note that even though $U_t$ is defined differently, we still have martingale differences and the key upper estimate of $U_t\leq \text{const}\cdot W_t^\phi$ still holds.} of Theorem \ref{SLLN} that for fixed  $s$ and $\epsilon>0$
\[
\lim_{t\uparrow\infty} P_{\delta_x}\left(\left| U_{t+s } [g]- E(U_{t+s }[g] |\mathcal{F}_t)\right|>\epsilon\right) =0.
\]
By the Markov inequality, for each $\epsilon>0$,
\[
P_{\delta_x}\left(\left|E(U_{t+s }[g]|\mathcal{F}_t) - \langle\phi g, \widetilde{\phi}\rangle W^\phi_t\right|>\epsilon\right)\leq \frac{1}{\epsilon}E_{\delta_x} \left|E(U_{t+s }[g]|\mathcal{F}_t) - \langle\phi g, \widetilde{\phi}\rangle W^\phi_t\right|.
\]
However, making use of the many-to-one identity, the right hand side above respects the following inequalities,
\begin{eqnarray*}
&&E_{\delta_x} \left|E(U_{t+s }[g]|\mathcal{F}_t) - \langle\phi g, \widetilde{\phi}\rangle W^\phi_t\right|\\
&&\leq E_{\delta_x}\left(\sum_{i=1}^{N_t} e^{-\lambda_c t}\phi(X_i(t)) |h_{s }(X_i(t)) - \langle\phi g,\widetilde{\phi}\rangle|\right)\\
&&=\phi(x) \mathbb{E}_x^\phi |h_{s }(\xi_t)  - \langle\phi g,\widetilde{\phi}\rangle|.
\end{eqnarray*}
Hence taking limits as $t\uparrow\infty$, and using ergodicity of the spine as well as the uniform boundedness of  $h_s(x)$, we have
\[
\lim_{t\uparrow\infty}P_{\delta_x}\left(\left|E(U_{t+s }[g]|\mathcal{F}_t) - \langle\phi g, \widetilde{\phi}\rangle W^\phi_t\right|>\epsilon\right) \leq \frac{\phi(x)}{\epsilon}
\langle|h_s - \langle\phi g,\widetilde{\phi}\rangle| , \phi\widetilde{\phi}\rangle.
\]
Finally, noting that $\lim_{s\uparrow\infty}h_s(x) =  \langle\phi g,\widetilde{\phi}\rangle$ and again the uniform boundedness of $h_s(x)$,  we have by dominated convergence that
\[
\lim_{s\uparrow\infty}\lim_{t\uparrow\infty}P_{\delta_x}\left(\left|E(U_{t+s}[g]|\mathcal{F}_t) - \langle\phi g, \widetilde{\phi}\rangle W^\phi_t\right|>\epsilon\right) \leq \frac{\phi(x)}{\epsilon} \langle\lim_{s\uparrow\infty}|h_s - \langle\phi g,\widetilde{\phi}\rangle| , \phi\widetilde{\phi}\rangle =0.
\]
To finish, we need an epsilon-delta argument. First note that, under the given conditions, the martingale $W^\phi$ converges in the $L^p$ norm and hence, in particular, converges in probability to its limit. Next, fix $\epsilon>0$. Then
\begin{eqnarray*}
&&P_{\delta_x}\left(\left| U_{t+s }[g] - \langle\phi g, \widetilde{\phi}\rangle W^\phi_\infty\right|>\epsilon\right)\\
&&\leq
P_{\delta_x}\left(\left| U_{t+s } [g]- E(U_{t+s}[g]|\mathcal{F}_t)\right|+\left| E(U_{t+s}[g]|\mathcal{F}_t)-\langle\phi g, \widetilde{\phi}\rangle W^\phi_t\right|\right.\\
&&\hspace{7cm}\left.+\langle\phi g, \widetilde{\phi}\rangle \left|W^\phi_t - W^\phi_\infty \right|
>\epsilon\right) \\
&&\leq P_{\delta_x}\left(\left| U_{t+s } [g]- E(U_{t+s}[g]|\mathcal{F}_t)\right|>\frac{\epsilon}{3}\right) +
P_{\delta_x}\left(\left| E(U_{t+s}[g]|\mathcal{F}_t)-\langle\phi g, \widetilde{\phi}\rangle W^\phi_\infty\right|>\frac{\epsilon}{3}\right)\\
&&\hspace{2cm}+P_{\delta_x}\left(  \langle\phi g, \widetilde{\phi}\rangle \left|W^\phi_t - W^\phi_\infty \right| > \frac{\epsilon}{3} \right).
\end{eqnarray*}
Now for each $\delta>0$, by choosing $s$ sufficiently large, we can take limits on the left hand side above to deduce
\[
\lim_{t\uparrow\infty}P_{\delta_x}\left(\left| U_{t+s }[g] - \langle\phi g, \widetilde{\phi}\rangle W^\phi_\infty\right|>\epsilon\right)<\delta.
\]
So, in fact
\[
\lim_{t\uparrow\infty}P_{\delta_x}\left(\left| U_{t } [g]- \langle\phi g, \widetilde{\phi}\rangle W^\phi_\infty\right|>\epsilon\right)=0.
\]
In particular,  taking  $g =  \kappa/\phi$ for any $\kappa\in C_c^+(D)$, yields the convergence (in probability) of $\exp \{-\lambda _{c}t\}X_{t}$  in the vague topology of measures.
\end{proof}
\textbf{Acknowledgement.} This research has been completed during the first author's visit to Bath and was funded by an EPSRC grant (EP/E05448X/1). The generous support of EPSRC is gratefully acknowledged. The first author also expresses his thanks to the University of Bath for the hospitality during his stay.

\end{document}